\documentclass[11 pt]{amsart}

\usepackage{times}
\usepackage{geometry}
\usepackage{amssymb}
\usepackage{latexsym,amssymb,  amsmath, amscd, amsfonts}
\usepackage{graphicx}
\usepackage[percent]{overpic}
\usepackage{pdfsync}
\usepackage{units}
\usepackage{hyperref}
\usepackage{euscript}
\usepackage{multicol}
\usepackage{epstopdf}
\usepackage{paralist}

\newtheorem{theorem}{Theorem}
\newtheorem{lemma}[theorem]{Lemma}

\newtheorem{corollary}[theorem]{Corollary}

\theoremstyle{definition}

\newtheorem{conjecture}[theorem]{Conjecture}
\newtheorem{const}[theorem]{Construction}
\newtheorem{remark}[theorem]{Remark}

\newcommand{\R}{\mathbb{R}}

\def\Lk{{\operatorname{Lk}}}

\def\Rib{{\operatorname{Rib}}}
\def\Cr{{\operatorname{Cr}}}
\def\Len{{\operatorname{Len}}}

\begin{document}

\title{Ribbonlength upper bounds  for small crossing knots and links}
\author[Z. Chen]{Zhicheng Chen}
\address{Henry Chen: Washington \& Lee University}
\email{chenh25@mail.wlu.edu}
\author[E. Denne]{Elizabeth Denne}
\address{Elizabeth Denne: Washington \& Lee University, Department of Mathematics, Lexington VA}
\email[Corresponding author]{dennee@wlu.edu}
\urladdr{https://elizabethdenne.academic.wlu.edu/}
\author[K. Patterson]{Kyle Patterson}
\address{Kyle Patterson: Washington \& Lee University}
\email{kopatterson@mail.wlu.edu}
\author[T. Patterson]{Timi Patterson}
\address{Timi Patterson: Washington \& Lee University}
\email{tpatterson@mail.wlu.edu}
\date{\today}
\makeatletter								
\@namedef{subjclassname@2020}{%
  \textup{2020} Mathematics Subject Classification}
\makeatother


\begin{abstract}
Given a thin strip of paper, tie a knot, connect the ends, and flatten into the plane. This is a physical model of a folded ribbon knot in the plane, first introduced by Louis Kauffman. We study the folded ribbonlength of these folded ribbon knots, which is defined as the knot's length-to-width ratio. The {\em ribbonlength problem} asks to find the infimal folded ribbonlength of a knot or link type. By finding new methods of creating folded ribbon knots, we improve upon existing upper bounds for the folded ribbonlength of $(2,q)$-torus links, twist knots, and pretzel links. These give the best known bounds to date for small crossing knots in these families.  For example, there is a folded ribbonlength twist knot $T_n$ with folded ribbonlength $\Rib(T_n) = n +6$. Applying this to the figure-eight knot $T_2$ yields a folded ribbonlength $\Rib(T_2)= 8$, which we conjecture is the infimum.
\end{abstract}

\maketitle

\section{Introduction to Folded Ribbon Knots and Links}\label{sect:intro}

We can imagine a folded ribbon knot to be constructed from a rectangular strip of paper which we tie into a knot, connect the ends, and then flatten into the plane.  A folded ribbon trefoil knot is shown in Figure~\ref{fig:trefoil}, and the rightmost figure shows what happens when the the ribbon knot  is``pulled tight''.
The pentagonal shape of this ``tight" trefoil knot is well known in recreational mathematics \cite{CR, John, Wel}. A closely related (but different) problem is that of multi-twist paper M\"obius bands and annuli. These are arc-length preserving embeddings of flat M\"obius bands or annuli into $\R^3$, while folded ribbon knots lie in the plane.  Richard Schwartz \cite{RES-Mob} recently proved a long standing conjecture of Ben Halpern and Charles Weaver \cite{HW}, that the infimal aspect ratio of a rectangular strip of paper needed to tie a paper M\"obius band is $\sqrt{3}$. It turns out that a folded ribbon knot can be well-approximated by an embedded paper M\"obius band or annulus. Thus information about folded ribbon unknots (for example in \cite{Den-FRLU, Den-TP, Hen}) has shed light on conjectures about the infimal aspect ratio of multi-twist embedded paper M\"obius bands and annuli.  Knotted ribbon shapes also appear in DNA structures and some of these are circular \cite{Flap, DNA1}. The folding of two and three-dimensional structures in ribbons has also appeared in robotics \cite{RobFold}, and in genetics with a ribosomal walking robot \cite{RiboRobot}.

\begin{center}
    \begin{figure}[htbp]
        \begin{overpic}[scale=1]{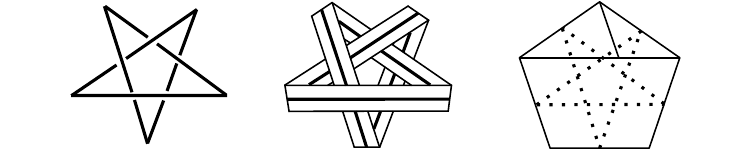}
        \end{overpic}
        \caption{Starting with a polygonal trefoil knot diagram (left), we create a folded ribbon trefoil knot (center). On the right, the pentagon shape arises when this particular folded ribbon trefoil knot is ``pulled tight''.}
        \label{fig:trefoil}
    \end{figure}
\end{center}

The idea of a folded ribbon knot was formalized by Lou Kauffman \cite{Kauf05} in 2005. Since the ribbon knot is folded, the corresponding knot diagram is polygonal. Polygonal links\footnote{We use standard the definitions and theory of knots and links as found in texts like \cite{Adams, Crom, JohnHen, Liv}.} are made of a finite number of straight line segments, colloquially called {\em sticks}.  Following Kauffman, we create a {\em folded ribbon link} by taking a polygonal link diagram $L$ and placing two parallel lines at an equal distance from each edge (with one on each side).   At each vertex of the link diagram, we add a fold line which is perpendicular to the {\em fold angle} $\theta$ at the vertex (here $0\leq\theta\leq \pi$). This results in a ribbon with width $w>0$,  denoted by $L_w$. We note that the fold lines act like a mirror that reflect adjacent edges of the link diagram into each other.  A comparison of a polygonal trefoil knot and its corresponding folded ribbon knot can be seen in Figure~\ref{fig:trefoil}. Since we are considering the folded ribbon link in the plane, we assign continuous over- and under-information to the ribbon which is consistent with the crossing information in the knot diagram. Complete details of this definition can be found in \cite{Den-FRS, Den-FRF, DKTZ}.  In addition, folded ribbon knots are topologically either a M\"obius band (when there are an odd number of sticks in the corresponding knot) or an annulus (when there are an even number of sticks). 

As illustrated on the right in Figure~\ref{fig:trefoil}, we can imagine taking a folded ribbon knot and ``pulling it tight". This is the same as asking how much ribbon is needed to tie a particular knot or link type. More formally, we are interested in finding the infimal folded ribbonlength over a link type. Here, the folded ribbonlength of a folded ribbon link \( L_w \) is defined as the ratio of the length of the link \( L \) to its width \( w \), and is denoted
$ \Rib(L_w) = \nicefrac{\Len(L)}{w}.$  The {\em folded ribbonlength problem} seeks to find the infimum of the ribbonlength over a link type $[L]$, that is
$$\Rib([L]) = \inf_{L_w \in [L]} \Rib(L_w).$$

The folded ribbonlength problem can also be viewed as a 2-dimensional version of the {\em ropelength problem} which asks for the minimum amount of rope needed to tie a knot in a rope of unit diameter.  (See for instance \cite{BS99,CKS,DDS,DE,gm,lsdr}.)

When computing folded ribbonlength it is helpful to use the local geometry and trigonometry of folds and crossings --- complete details can be found in \cite{Den-FRF}.  Observe that locally, a fold in a piece of ribbon consists of two overlapping triangles, while a crossing consists of two overlapping parallelograms.  The following remark will be used often in our ribbonlength computations.
\begin{remark}\label{rmk:length} 
The ribbonlength of any square ribbon segment is $1$. If we fold a piece of ribbon so that the fold angle $\theta = \nicefrac{\pi}{2}$, we get two overlapping isosceles right triangles, each with a ribbonlength of $\nicefrac{1}{2}$. This $\nicefrac{\pi}{2}$-fold has a total ribbonlength of $1$. \end{remark}

When seeking to find the infimal folded ribbonlength for a knot or link type, we can find an upper bound simply by computing the folded ribbonlength for one geometric realization of the knot.  For example, Kauffman \cite{Kauf05} showed that the folded ribbon trefoil knot $K$ in Figure~\ref{fig:trefoil}  has $\Rib(K_w)= 5 \cot\left(\frac{\pi}{5}\right) \leq 6.89$. However, this is not the shortest trefoil knot! The second author with different coauthors \cite{Den-FRF} have discovered two different ways of constructing the trefoil knot such that $\Rib(K_w)=6$. 

Many authors \cite{Den-FRF, Den-TP, Kauf05, KMRT, KNY-2Bridge, KNY-TwTorus, Tian-A} have worked on the folded ribbonlength problem over the years. Here is a summary of the best upper bounds of folded ribbonlength to date.

\begin{enumerate}
\item[(a)] Any trefoil knot $K$ has $\Rib([K])\leq 6$ from \cite{Den-FRF}.
\item[(b)] Any $(p,q)$-torus link has $\Rib([T(p,q)])\leq 2p$ where $p\geq q\geq 2$ from \cite{Den-FRF}.
\item[(c)] Any $(2,q)$-torus knot has 
$$\Rib([T(2,q)])\leq\begin{cases} 8\sqrt{3}\leq 13.86 & \text{for $q$ odd from \cite{Den-TP},}\\
2q & \text{ for $q$ odd from \cite{Den-FRF}}.
\end{cases}
$$
\item[(d)] Any figure-8 knot has $\Rib([K])\leq 10$ from \cite{Den-FRF}.
\item[(e)] Any twist knot $T_n$, with $n$ half-twists has 
$$\Rib([T_n])\leq \begin{cases} 9\sqrt{3}+2 \leq 17.59 & \text{ when $n$ is odd from \cite{Den-TP},} \\
8\sqrt{3} +2 \leq  15.86 & \text{ when $n$ is even from \cite{Den-TP},}\\
2n+6  \quad & \text{ for $n\leq 5$ from \cite{Den-FRF}}.
 \end{cases}$$
 \item[(f)] Any $(p,q,r)$-pretzel link $L$ has $\Rib([L])\leq 2(|p|+|q|+|r|) + 2$ from \cite{Den-FRF}.
\item[(g)] Any $2$-bridge knot with crossing number $\Cr(K)$ has $\Rib([K])\leq  2\Cr(K)+2$ from \cite{KNY-2Bridge}.
\item[(h)] Any twisted torus knot\footnote{A twisted torus knot $T_{p,q;r,s}$ is obtained from a torus knot $T_{p,q}$  by twisting $r$ adjacent strands $s$ full twists.} $K=T_{p,q;r,s}$ (from \cite{KNY-TwTorus}) has
$$\Rib([K]) \leq \begin{cases} 2(\max\{p,q,r\}+|s|r) \quad \text{for $0<r<p+q$ and $s\neq 0$,}
\\ 2(p+(|s|-1)r) \quad \text{for $0< r\leq p-q$ and $s\neq 0$.}
\end{cases}
$$
\end{enumerate}

In this paper, we will improve the known upper bounds on (c) the $(2,q)$ torus links, (e)  the twist knots, and (f) the pretzel links.  See Theorems~\ref{thm:torus}, \ref{thm:twist}, and \ref{thm:pretzel}. The key point is that these bounds are the best known for knots with small crossing number\footnote{The crossing number $\Cr(L)$ of a link $L$ is the minimum number of crossings in any knot diagram of $L$  see \cite{Adams, Crom, Liv}.} for twist knots and $(2,q)$-torus links.

Observe that the ribbon gives a framing around the link diagram and we can study this in addition to the link type when considering folded ribbonlength. However, we will not pursue this approach here. Instead we note that there have been a number of results about framed unknots in \cite{RES-Brown, Den-FRLU, Den-TP, Hen,RES-Mont, RES-Mob2, RES-Mob}. The framed nontrivial knot case remains open.

The relationship between the infimal folded ribbonlength of a knot type $[K]$ and its crossing number $\Cr(K)$ is a well known problem in the study of folded ribbon knots. The {\em ribbonlength crossing number problem} seeks to establish constants \( c_1 \), \( c_2 \), \( \alpha \), and \( \beta \) such that the following inequality holds for all knot types:
\begin{equation}
c_1\cdot \Cr(K)^\alpha \leq \Rib([K]) \leq c_2\cdot \Cr(K)^\beta.
\label{eq:crossing}
\end{equation}
In the early 2000s, Yuanan Diao and Rob Kusner conjectured that \( \alpha = \frac{1}{2} \) and \( \beta = 1 \).  Many attempts to prove this conjecture have been made over the years (see \cite{Den-FRC, KMRT, Tian-A}). The ribbonlength crossing number problem has recently been completely solved!   In 2024, Hyoungjun Kim, Sungjong No, and Hyungkee Yoo~\cite{KNY-Lin} gave a solution for the upper bound  by proving that that $\beta\leq 1$. Specifically, they proved that for any knot or link, 
\begin{equation}\Rib([K])\leq 2.5\Cr(K)+1.
\label{eq:bound}
\end{equation}
In 2025, the second and fourth authors \cite{Den-TP} gave a solution for the lower bound by proving that $\alpha=0$. As stated above in (c) and (e), we proved that the $(2,q)$-torus knots and the twist knots have a uniform upper bound for folded ribbonlength, while the crossing numbers are unbounded. 

We make two observations. Firstly, while Equation~\ref{eq:bound} gives an upper bound for all knots, the bounds given above (a)-(f) and Theorems~\ref{thm:torus}, \ref{thm:twist}, and \ref{thm:pretzel} are usually much better.  Secondly, while there are uniform upper bounds on folded ribbonlength in (c) and (e) above, these are not the best for small crossing knots.  This is the central theme of this paper.

In Section~\ref{sect:knots}, we both introduce and review basic facts about the $(2,q)$-torus link, twist knot, and $P(p,q,r)$ pretzel link families. These three knot and link families are constructed using half-twists joined in particular ways. In Section~\ref{sect:wrap}, we give a technique in Construction~\ref{const:wrap} for constructing $n$ half-twists in an efficient way. We then apply this to $(2,q)$-torus links. In Theorem~\ref{thm:torus} we prove that any $(2,q)$-torus link can be constructed so that its folded ribbonlength $\Rib(T(2,q)_w)=q+3$.  In Section~\ref{sect:twist} we apply Construction~\ref{const:wrap} to twist knots. In Theorem~\ref{thm:twist} we prove that any $T_n$ twist knot can be constructed so that its folded ribbonlength $\Rib((T_n)_w)=n+6$. This immediately proves Corollary~\ref{cor:fig8} that a folded ribbon figure-8 knot $K$ can be constructed with folded ribbonlength $\Rib(K_w)=8$. This is the best known bound for the figure-8 knot to date!   Finally, in Section~\ref{sect:pretzel}, we modify the method of constructing $n$ half-twists in Construction~\ref{const:pretzel-wrap}. These modified half-twists are then used to construct folded ribbon pretzel links. In Theorem~\ref{thm:pretzel} we prove that any $P(p,q,r)$ pretzel link can be constructed so that its folded ribbonlength satisfies $\Rib(P(p,q,r)_w)=|p|+|q|+|r|+6$.

We encourage the reader to locate some long strips of paper with which they can recreate the constructions found in this paper. The constructions will make more sense with a physical model in hand.


\section{Half-twists and link families}\label{sect:knots}

We are interested in studying the folded ribbonlength of some of the infinite families of knots and links that can be constructed from {\em half-twists}, which are illustrated in Figure~\ref{fig:half-twists-torus}. A sequence of half-twists in a knot or link diagram can be replaced by a box containing an integer. The sign of the integer indicates the orientation of the half-twists. Here, the integer is positive\footnote{There is a choice here, and other texts may use the opposite convention.}  when the bottom ends of the pair of strings have been twisted in a counterclockwise direction.  Families of knots and links can be drawn using these twist-boxes. These families have been well-studied (see for instance \cite{Adams, Crom, JohnHen, Liv}).

\begin{center}
\begin{figure}[htbp]
\begin{overpic}{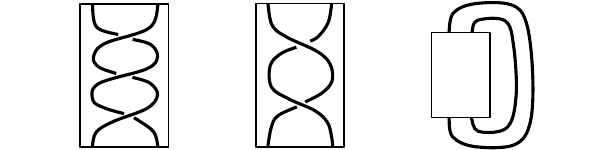}
\put(8,11){$+3$}
\put(37,11){$-2$}
\put(76, 12){$q$}
\end{overpic}
\caption{A $+3$ twist box on the left, a $-2$ twist-box in the center, and a $(2,q)$-torus link on the right.}
\label{fig:half-twists-torus}
\end{figure}
\end{center}

Recall that a {\em torus link} is a link which can be embedded on a torus. A $(p,q)$-torus link, denoted $T(p,q)$, is classified by the number of ways the link winds around the torus the long way ($p$), and the number of ways it winds around the torus the short way ($q$). In this paper we focus on $(2,q)$-torus links as illustrated on the right in Figure~\ref{fig:half-twists-torus}. We know, for example, the crossing number of a $(2,q)$-torus link is $\Cr(T(2,q))= q$. We also know that $T(p,q)$ torus links are equivalent to $T(q,p)$ torus links.  The greatest common divisor of $p$ and $q$ corresponds to the number of components in the link, thus $T(p,q)$ is a knot if and only if $\gcd(p,q)=1$. Note that the trefoil knot is a $T(2,3)\cong T(3,2)$ torus knot.

A {\em twist knot} with $n$ half-twists, denoted by $T_n$, is constructed from $n$ half-twists, and then the ends of the half-twists are linked together in clasp to form a closed loop\footnote{Equivalently, a twist knot is a Whitehead double of an unknot.}. The orientation of the clasp is chosen so that the corresponding knot diagram is alternating. An example of a twist knot with 5 half-twists is shown in Figure~\ref{fig:tw-knots-box}. Like torus knots, twist knots have been well studied. We know that twist knots are achiral and that the crossing number of a twist knot is $\Cr(T_n)=n+2$. Note that the trefoil knot is a $T_1$ twist knot and the figure-8 knot is a $T_2$ twist knot. 

\begin{center}
\begin{figure}[htbp]
\begin{overpic}{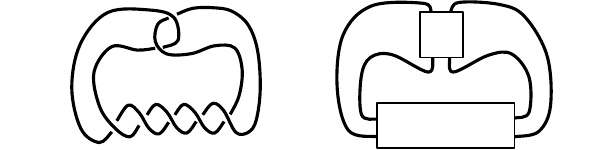}
\put(23,25){Clasp}
\put(21,-1){Half-twists}
\put(71,18){$-2$}
\put(73, 1){\rotatebox{90}{$+n$}}
\end{overpic}
\caption{A $T_5$ twist knot on the left, and an arbitrary $T_n$ twist knot on the right.}
\label{fig:tw-knots-box}
\end{figure}
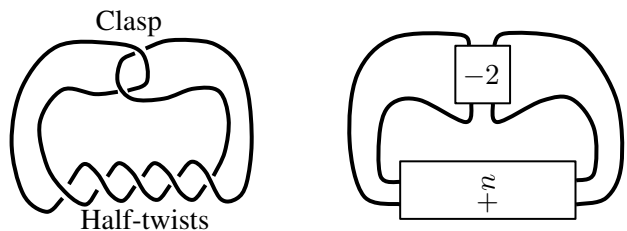
\end{center}

A {\em 3-strand pretzel link} $P(p,q,r)$ is made of three pairs of $p$, $q$, and $r$ half-twists (called strands) which are joined as shown in Figure~\ref{fig:pretzel-box}. The pretzel links are a complex family and contain many other families of knots and links, for example the twist knots $T_n$ as $P(n,1,1)$. The crossing number and the number of components of a $P(p,q,r)$ pretzel link both depend on the sign and parity of $p$, $q$, and $r$. For example, when $p$, $q$, and $r$ are all odd and have the same sign, the knot diagram is reduced and alternating. This means that the crossing number is $\Cr(P(p,q,r))=|p|+|q|+|r|$. (In general, we can only know that $\Cr(P(p,q,r))\leq |p|+|q|+|r|$.) We know that if zero or one strands have an even number of half-twists, then $P(p,q,r)$ is a knot.  Otherwise, $P(p,q,r)$ is a link of two or three components. 

\begin{center}
    \begin{figure}[htbp]
        \begin{overpic}{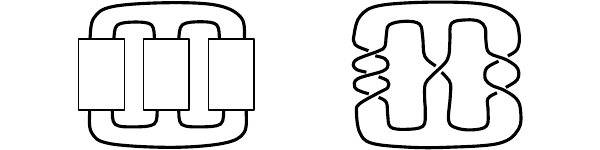}
        \put(16,12){$p$}
        \put(27,12){$q$}
        \put(38,12){$r$}
        \end{overpic}
        \caption{On the left, an arbitrary $P(p,q,r)$ pretzel link, and on the right a $P(3, 1, -2)$ pretzel knot.}
        \label{fig:pretzel-box}
    \end{figure}
\end{center}

The definition of 3-strand pretzel links can also be generalized to define $n$-strand pretzel links, made with $(p_1,p_2,\dots,p_n)$ half-twists. Montesinos links also generalize pretzel links, but we do not discuss these here. Finally, we note that while collections of half-twists are used to construct 2-bridge knots and links, we do not discuss this family here. 



\section{The wrap method applied to $(2,q)$ torus links}\label{sect:wrap}

In this section, we will describe a new construction method, called the {\em wrap method}, which helps optimize the folded ribbonlength of knots that consist of half-twists. In this section, we will apply the wrap method to $T(2,q)$-torus links. In Section~\ref{sect:twist}, we will apply the wrap method to twist knots. In Section~\ref{sect:pretzel}, we will apply a modified version of this method to pretzel links.

In order to find an efficient method of creating half-twists, we change our viewpoint. Instead of viewing a half-twist as two strands twisting around one another as shown on the left in Figure~\ref{fig:half-twists-new}, we make one strand ($CD$) straight and coil or wrap the second ($AB$) around the first. Unlike Figure~\ref{fig:half-twists-new} right, we can create coils which overlap one another.  This leads to the difficulty of locating the end of ribbon inside the coil. We solve this by simply adding a $\pi/2$ fold to the inside end.

\begin{center}
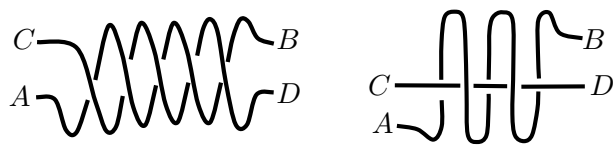
\begin{figure}[htbp]
\begin{overpic}{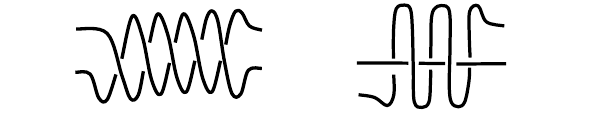}
\put(9,5.5){$A$}
\put(44,13){$B$}
\put(9.5,13){$C$}
\put(44,6){$D$}
\put(56.5,2){$A$}
\put(84,13){$B$}
\put(56,7){$C$}
\put(85,7){$D$}
\end{overpic}
\caption{Two different viewpoints of half-twists.}
\label{fig:half-twists-new}
\end{figure}
\end{center}

\begin{const}[Wrap Method] \label{const:wrap}
To construct $n$ half-twists (assuming $n\neq 0$), begin with two pieces of ribbon labeled $AB$ and $CD$. 
\\ {\bf Case 1:} Let $n$ be a positive integer. 
\\ Step 1: Place ribbon $AB$ over the top of ribbon $CD$, as shown in the first image in Figure~\ref{fig:wrap}. 
\\ Step 2: Fold end $B$ downwards with fold angle $\pi/2$, as shown in the second image in Figure~\ref{fig:wrap}. 
\\ Step 3: Wrap end $B$ under ribbon $CD$ then fold upwards with fold angle $0$. This means end $B$ now points upward at the back, as shown in the third image in Figure~\ref{fig:wrap}. This creates the first half-twist. 
\\ Step 4: Wrap end $B$ over ribbon $CD$ and end $A$, and fold down with fold angle $0$  as shown in the right-most image in Figure~\ref{fig:wrap}. This creates the second half-twist.
\\ Step 5: Repeat Steps 3 and 4 as often as needed to construct any odd or even number of half-twists.  Note that if we view the top of the twists as being located on the left, then the half-twists created are positive half-twists.

 \begin{center}
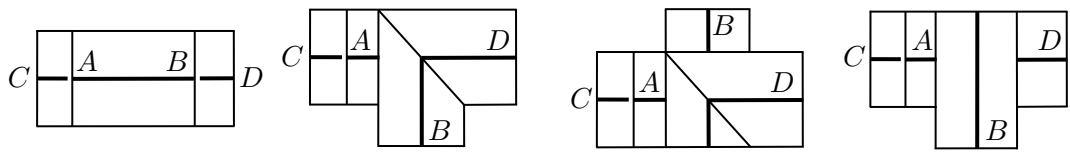
\begin{figure}[htbp]
\begin{overpic}{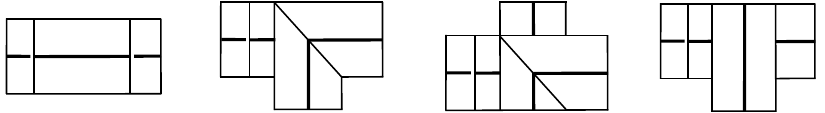}
\put(4.5,7.5){$A$}
\put(13,7.5){$B$}
\put(-2,6){$C$}
\put(20,6){$D$}
\put(30.5,9.5){$A$}
\put(38,1){$B$}
\put(24,8){$C$}
\put(43.5,9.5){$D$}
\put(58,5.5){$A$}
\put(65,11){$B$}
\put(51.5,4){$C$}
\put(70.5,5.5){$D$}
\put(84,9.5){$A$}
\put(91,1){$B$}
\put(77.5,8){$C$}
\put(96,9.5){$D$}
\end{overpic}
\caption{Constructing half-twists using the wrap method.}
\label{fig:wrap}
\end{figure}
\end{center}

{\bf Case 2:} Let $n$ be a a negative integer. Follow the steps for Case 1, but in Step~2, fold ribbon end $B$ upwards (rather than downwards) with fold angle $\pi/2$. Steps 3, 4, 5 continue in a similar way. 
\qed
\end{const}

In Construction~\ref{const:wrap}, we observe that for an odd number of positive half-twists, ribbon end $B$ will be pointing upwards at the back, while for an even number of positive half-twists, ribbon end $B$ will be pointing downwards at the front. 

\begin{lemma} \label{lem:twist}
In the wrap method Construction~\ref{const:wrap}, we see $n$ half-twists can be constructed using at least $|n|+2$ units of ribbonlength and the corresponding knot diagram has $|n|+3$ sticks.
\end{lemma}

\begin{proof}
As seen in Figure~\ref{fig:wrap}, the half-twists created by the wrap method are concentrated in a square region. We measure the ribbonlength in this square region using Remark~\ref{rmk:length}, and exclude the ends of the two pieces of ribbon. 

Assume $n$ is a positive integer. From the Step 2 (second picture in Figure~\ref{fig:wrap}), the $\pi/2$-fold of end $B$ downwards is created using two overlapping isosceles right triangles, and contributes one unit of ribbonlength and two sticks. In Step 3 end $B$ wraps around the back of $CD$ and contributes a square of ribbonlength, that is one unit of ribbonlength and one stick. In Steps 4, and 5, each additional wrap of end $B$ around $CD$ contributes another unit of ribbonlength and one stick.

Finally, the part of ribbon $CD$ (a square) that is wrapped by ribbon $B$ contributes $1$ unit of ribbonlength and $1$ stick. Thus we need  $3$ units of ribbonlength and $4$ sticks to construct one half-twist. Each subsequent half-twist from Steps 4 and 5 contributes one unit of ribbonlength and one stick. 
If $n$ is a negative integer, the same computations hold.  Hence, creating $n$ half-twists uses at least  $|n|+2$ units of ribbonlength and  $|n|+3$ sticks. 
 \end{proof}

\subsection{Ribbonlength of $(2,q)$-torus links}
Recall from Section~\ref{sect:knots}, that a $(2,q)$-torus link is constructed from $q$ half-twists joined in a specific way  (see Figure~\ref{fig:half-twists-torus}). In this section, we assume without loss of generality that $q$ is a positive integer. Looking now at Figure~\ref{fig:tw-odd-even}, this means that the top two strands of the  $n$ half-twists are joined together and the bottom two strands are joined together (regardless of whether $q$ is odd or even). Specifically, when $q$ is odd (on the left in  in Figure~\ref{fig:tw-odd-even}), we join ends $C$ and $B$ together and ends $A$ and $D$ together, and form a torus knot. When $q$ is even (on the right in  in Figure~\ref{fig:tw-odd-even}) we join ends $C$ and $D$ together and ends $A$ and $B$ together, and form a torus link of two components.  We now show how to construct folded ribbon $(2,q)$-torus links using the wrap method for the $q$ half-twists from Construction~\ref{const:wrap}.

\begin{center}
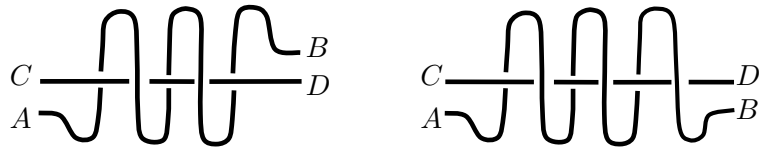
\begin{figure}[htbp]
\begin{overpic}{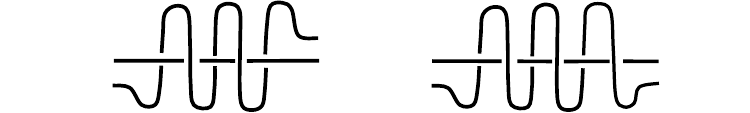}
\put(12, 2){$A$}
\put(43,9.5){$B$}
\put(12,6.5){$C$}
\put(43,5.5){$D$}
\put(55,2){$A$}
\put(88,3){$B$}
\put(55,6.5){$C$}
\put(88,6.5){$D$}
\end{overpic}
\caption{Half-twists used to construct $(2,q)$-torus knots (left) and links (right).}
\label{fig:tw-odd-even}
\end{figure}
\end{center}

\begin{const}[Folded ribbon $(2,q)$-torus knots]  \label{const:torus-odd}
To construct a folded ribbon $(2,q)$-torus knot, begin with two pieces of ribbon labeled $AB$ and $CD$, and place ribbon $AB$ on top of ribbon $CD$. \\
Step 1: Use the wrap method from Construction~\ref{const:wrap} to construct $q>0$ odd half-twists in the two pieces of ribbon. This is shown on the left in Figure~\ref{fig:torus-odd}. Recall that the half-twists are concentrated in a square region, and end $B$ is pointing up at the back.
\\ Step 2: Fold $B$ at the back to the left with fold angle $\pi/2$. If we flip the ribbon over, we see the center image in  Figure~\ref{fig:torus-odd} where $B$ appears to the right over end $C$. In this image, end $A$ is below ends $B$ and $C$ and is not shown. Then join ends $B$ and $C$ together with a fold line. 
\begin{center}
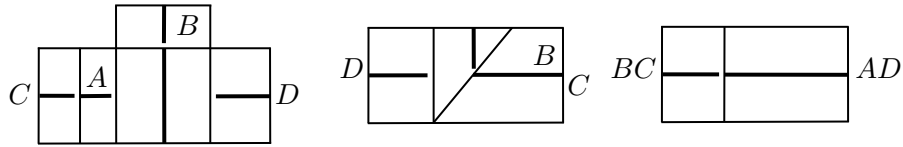
\begin{figure}[htbp]
\begin{overpic}{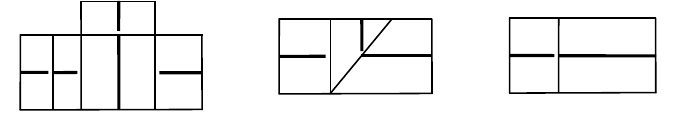}
\put(8.5, 7){$A$}
\put(19,13){$B$}
\put(-0.5,5){$C$}
\put(30.5,5){$D$}
\put(60.5,9.5){$B$}
\put(64.5,6){$C$}
\put(38,8){$D$}
\put(69.5,8){$BC$}
\put(98,8){$AD$}
\end{overpic}
\caption{The three steps to construct a folded ribbon $(2,q)$-torus knot from Construction~\ref{const:torus-odd}.}
\label{fig:torus-odd}
\end{figure}
\end{center}

\noindent Step 3:  Starting with the ribbon in the original orientation, fold end $A$ over to the right with fold angle $0$ so that it lies over end $D$. This is shown on the right in Figure~\ref{fig:torus-odd}. Then join ends $A$ and $D$ together with a fold line. 
\\ Step 4: Shorten the joined ends $BC$ and $AD$ so they lie against the square region created by the half-twists. The completed torus knot has the shape of a square, so we have not illustrated this here. 
\qed
\end{const}

\begin{const}[Folded ribbon $(2,q)$-torus links] \label{const:torus-even}
To construct a folded ribbon $(2,q)$-torus link, begin with two pieces of ribbon labeled $AB$ and $CD$, and place ribbon $AB$ on top of $CD$.  \\
Step 1: Use the wrap method from Construction~\ref{const:wrap} to construct $q>0$ even half-twists in the two pieces of ribbon. This is shown on the left in Figure~\ref{fig:torus-even}. Once again, the half-twists are concentrated in a square region, and end $B$ is pointing downwards at the front.
\\ Step 2:  Fold end $B$ over to the left with fold angle $\pi/2$ so that end $B$ lies over end $A$. This is shown in the center image in Figure~\ref{fig:torus-even}. Then join ends $A$ and $B$ together with a fold line.

\begin{center}
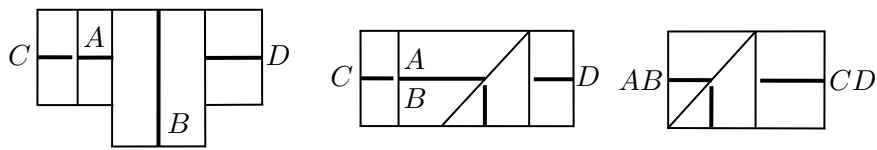
\begin{figure}[htbp]
\begin{overpic}{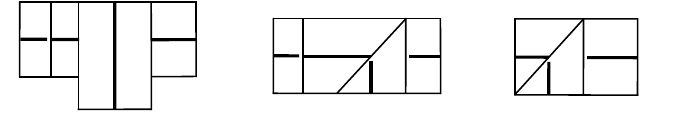}
\put(8.25, 12){$A$}
\put(18,2){$B$}
\put(-0.5,10){$C$}
\put(29.5,10){$D$}
\put(45.5,9.5){$A$}
\put(45.5,5){$B$}
\put(37,7.5){$C$}
\put(65.5,7.5){$D$}
\put(70.5,7){$AB$}
\put(95,7){$CD$}
\end{overpic}
\caption{The three steps to construct a folded ribbon $(2,q)$-torus link from Construction~\ref{const:torus-even}.}
\label{fig:torus-even}
\end{figure}
\end{center}

\noindent Step 3: Fold end $C$ at the back to the right with fold angle $0$ so that it lies under end $D$. This is shown on the right in Figure~\ref{fig:torus-even}. Then join ends $C$ and $D$ together with a fold line
\\ Step 4: Shorten the joined ends $AB$ and $CD$ so they lie against the square region created by the half-twists. The completed torus link has the shape of a square, so we have not illustrated this here. 
\qed
\end{const}

\begin{theorem}\label{thm:torus}
Any $(2,q)$-torus link type $L$ contains a folded ribbon link $L_w$ so that its folded ribbonlength is  $\Rib(L_w) = q+3$. The corresponding link diagram has $q+5$ sticks.
\end{theorem}

\begin{proof} 
Recall we assumed that $q$ is a positive integer. We compute the folded ribbonlength of the $(2,q)$-torus knots and links constructed in Construction~\ref{const:torus-odd} and Construction~\ref{const:torus-even}. We will again use Remark~\ref{rmk:length} in our ribbonlength computations.

In Step 1 of both constructions, we start with $q$ half-twists which are concentrated in a square region. We know from Lemma \ref{lem:twist} that these $q$ half-twists are constructed from $q+2$ units of ribbonlength (plus the ribbonlength of the ends of the pieces of ribbon).  The corresponding link diagram has $q+3$ sticks.

In Step 2 of both constructions, we take one of the existing squares from the half-twists and fold this in half with a $\pi/2$ fold. The end of this is then joined to one of the other ends and shortened. Thus no extra units of ribbonlength are added by this move. However, the fold adds one extra stick.

In Step 3 of both constructions, we take one end and fold it across the square region of half-twists. This is then joined to one of the other ends and shortened. In this step, one additional unit of ribbonlength is added and one extra stick.

Therefore, in both constructions we use $q+2+1=q+3$ units of folded ribbonlength, and $q+3+1+1=q+5$ sticks.
\end{proof}

 We now put Theorem~\ref{thm:torus} in context. Firstly, we note that this is the first time the $(2,q)$-torus link case has been proven.  Secondly, recall that the crossing number of a $(2,q)$-torus link is $\Cr(T(2,q))=q$. We first 
 observe that Theorem~\ref{thm:torus} greatly improves the previous best upper bound for small crossing ($q\leq 5$) knots of  $\Rib[T(2,q)])\leq 2q$ from \cite{Den-FRF}.  Thirdly, another immediate application of Theorem~\ref{thm:torus} is that the trefoil knot $T(2,3)$ has folded ribbonlength $\Rib(T(2,3)_w)=6$. This is the third construction of a trefoil knot with this ribbonlength --- the other two constructions can be found in \cite{Den-FRF}. Once again, we see that the geometric configuration of the trefoil in Figure~\ref{fig:trefoil} does not gives the minimum ribbonlength. However that trefoil knot is made of 5 sticks, so is a topological M\"obius band, whereas Construction~\ref{const:torus-odd} uses 8 sticks for the trefoil so is a topological annulus. The additional sticks used give more ``wiggle room" and allow the ribbonlength to decrease.   We have previously conjectured \cite{Den-TP} that the geometric configuration of the trefoil knot in Figure~\ref{fig:trefoil} has the infimal folded ribbonlength when the ribbon is a topological M\"obius band. Finally, we note that the second and fourth authors in \cite{Den-TP} have proved that for any $(2,q)$-torus knot, the infimal folded ribbonlength satisfies $\Rib([T(2,q)])\leq 8\sqrt{3}\leq 13.86$. This is a uniform upper bound on the folded ribbonlength. The proof relied on the fact the $q$ half-twists can be constructed with a finite amount of folded ribbonlength for any $q$, and then there is a finite amount of ribbon needed to join the ends of the half-twists. A similar uniform upper bound thus holds when $q$ is even, but we did not choose to compute this bound.

The key point here is that the folding techniques which give uniform upper bounds are not the same as the folding techniques needed for small crossing knots. Thus, if we ignore the topological type of the ribbon, we conjecture the following.

\begin{conjecture} \label{conj:torus}
For any $(2,q)$-torus link, there is a constant $C>0$, such that the infimal folded ribbonlength 
$$\Rib([T(2,q)])\leq \begin{cases} q+3 & \text{ when $q\leq 10$,}\\
8\sqrt{3}\leq 13.86 & \text{ when $q\geq 11$ is odd}, 
\\ C & \text{ when $q\geq 12$ is even}, 
\end{cases}
$$
\end{conjecture}

Another application of Theorem~\ref{thm:torus} is to the ribbonlength crossing number problem. We know immediately that  the folded ribbonlength $\Rib([T(2,q)]) \leq \Cr(T(2,q))+3$. This upper bound is much lower than the universal upper bound for any knot type given in Equation~\ref{eq:bound} of $\Rib([K])\leq 2.5\Cr(K)+1$. In addition, in  \cite{Den-TP} we noted that the uniform upper bound of 13.86 on folded ribbonlength when $q\geq 11$ means that $\alpha =0$ in Equation~\ref{eq:crossing}. That is for all knots $K$, there is a constant $c_1>0$ such that $c_1\cdot \Cr(K)^0\leq\Rib([K])$. This gave the lower bound for the ribbonlength crossing number conjecture.

Finally, we observe that Theorem~\ref{thm:torus} gives some insight into two component links.  Recall that the linking number of an oriented two component link $A\cup B$ is one half of the sum of the signs of the crossings between components $A$ and $B$ in any link diagram of $A\cup B$ (see for instance \cite{Adams, JohnHen}). It is not too hard to show that the linking number of a $(2,2n)$-torus link is $\Lk(T(2,2n))=n$. Our prior results about folded ribbonlength thus gives us insight into the folded ribbonlength of any two component link with linking number $n$ (regardless of link type).

\begin{conjecture}\label{conj:2link} 
For any two component link $L$ with linking number $|\Lk(L)|=n$, there is a constant $C>0$, such that the infimal folded ribbonlength 
$$\inf_{\substack{L\in[L], \\ |\Lk(L)|=n}} \Rib(L_w) \leq \begin{cases} 2n+3 & \text{ when $n\leq 5$}, 
\\ C & \text{ when $n\geq 6$.}
\end{cases}$$
\end{conjecture}

\section{The wrap method applied to twist knots }\label{sect:twist}
Recall from Figure~\ref{fig:tw-knots-box} in Section~\ref{sect:knots}, that a twist knot $T_n$ is constructed from $n$ half-twists and a clasp. In this section, we assume $n$ is a natural number (so $n>0$). We now show how to construct folded ribbon twist knots using the wrap method in Construction~\ref{const:wrap} to construct the $n$ half-twists.

\begin{const}[Folded ribbon twist knots] \label{const:twist}
To construct a folded ribbon $T_n$ twist knot, begin with two pieces of ribbon labeled $AB$ and $CD$, and place ribbon $AB$ on top of $CD$. There are two cases:  when $n$ is even and $n$ is odd. 
\\ {\bf Case 1.} Let $n>0$ be an even integer. 
\\ Step 1:  Use the wrap method from Construction~\ref{const:wrap} to construct $n$ half-twists in the two pieces of ribbon. This is shown on the left in Figure~\ref{fig:tw-even-1}. Recall that the half-twists are concentrated in a square region and end $B$ is pointing downwards at the front. We next join ends $A$ and $C$ and join ends $B$ and $D$ in a way that creates the clasp.
\\ Step 2:  Fold end $A$ over to the right with fold angle $0$ so that it lies over ribbon $B$ and end $D$, as shown in the center image in Figure~\ref{fig:tw-even-1}. 

\begin{center}
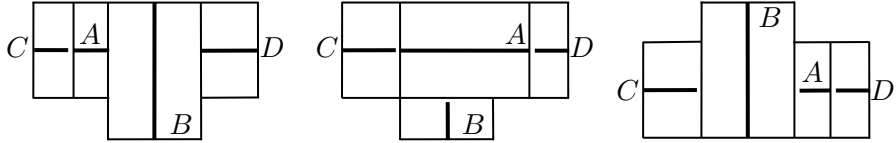
\begin{figure}[htbp]
\begin{overpic}{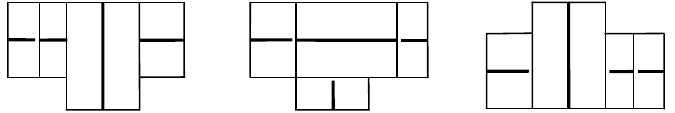}
\put(6.5, 11.5){$A$}
\put(17,1){$B$}
\put(-2,10){$C$}
\put(27.5,10){$D$}
\put(56,11.5){$A$}
\put(51,1){$B$}
\put(34,10){$C$}
\put(63.5,10){$D$}
\put(90.5,7){$A$}
\put(85.5,13.5){$B$}
\put(69,5){$C$}
\put(98.5,5){$D$}
\end{overpic}
\caption{Steps of creating the clasp region for $T_n$.}
\label{fig:tw-even-1}
\end{figure}
\end{center}

\noindent Step 3: Fold end $B$ up with fold angle $0$ so that it lies over ribbon $A$ as shown on the right in Figure~\ref{fig:tw-even-1}. 
\\ Step 4: Fold end $C$ over to the right fold angle $0$ so that it lies over ribbon $B$ and ends $A$ and $D$. Then join ends $A$ and $C$ together with a fold line, as shown on the left in Figure~\ref{fig:tw-even-2}. Shorten end $AC$ so that this fold line lies against the square region created by the half-twists.
\\ Step 5: To clearly see the way we join ends $B$ and $D$, flip the construction over as shown in the center image in Figure~\ref{fig:tw-even-2}. Note that end $AC$ is under end $D$.
\\ Step 6: Fold end $D$ over to the right with fold angle $0$, then fold end $D$ up with fold angle $\pi/2$. End $D$ now lies over end $B$ as shown on the right in Figure~\ref{fig:tw-even-2}. 
 Join ends $B$ and $D$ together with a fold line,  then shorten so that the fold line lies against the square region. The completed twist knot has the shape of a square, and is not shown here.

\begin{center}
\begin{figure}[htbp]
\begin{overpic}{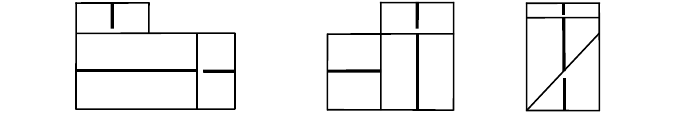}
\put(23.5, 7){$AC$}
\put(18,13){$B$}
\put(35,5){$D$}
\put(63,13){$B$}
\put(45,5){$D$}
\put(72,5){$AC$}
\put(84.5,17){$B$}
\put(84.5,11){$D$}
\end{overpic}
\caption{Steps of joining the clasp region for $T_n$.}
\label{fig:tw-even-2}
\end{figure}
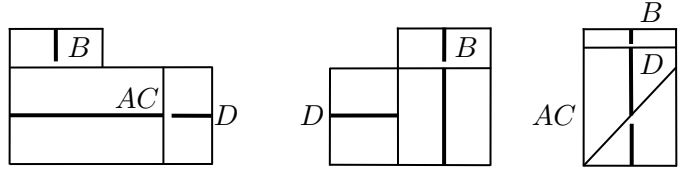
\end{center}

\noindent {\bf Case 2.} Let $n>0$ be an odd integer. The construction is very similar to the $n$ is even case, so we only give an outline here. Use Construction~\ref{const:wrap} to construct $n$ half-twists and note that end $B$ is pointing upwards at the back. Steps 2, 3, and 4 are the similar, except the folding occurs at the back. (End $C$ is folded to the right at the back, then end $B$ is folded downwards at the back, then end $A$ is folded to the right at the back. Ends $A$ and $C$ are then joined with a fold line.) 

In Steps 5 and 6, end $D$ is folded to join end $B$ in the front. Here, we fold $D$ over to the left with fold angle $0$, then fold end $D$ down with fold angle $\pi/2$. Ends $B$ and $D$ are then joined with a fold line. As before, we shorten ends $AC$ and $BD$ so their fold lines lie against the square region formed by the $n$ half-twists.
\qed
\end{const}

\begin{theorem} \label{thm:twist} 
Any twist knot type $T_n$ contains a folded ribbon twist knot $(T_n)_w$ with folded ribbonlength  $\Rib((T_n)_w) = n+6$. The corresponding knot diagram has $n+8$ sticks.
\end{theorem}

\begin{proof}
We compute the folded ribbonlength of the folded ribbon twist knot constructed in Construction~\ref{const:twist}. We again use Remark~\ref{rmk:length} in our computations.

In Step 1 of both cases, we start with $n$ half-twists which are concentrated in a square region. We know from Lemma \ref{lem:twist} that these $n$ half-twists are constructed from $n+2$ units of ribbonlength (plus the ribbonlength of the ends of the pieces of ribbon).  The corresponding link diagram has $n+3$ sticks.

In Steps 2, 3 and 4 of both cases we see that folding ribbon $B$ in between ribbons $A$ and $C$ uses $3$ units of ribbonlength ($3$ squares of ribbon) and $3$ sticks.  In Steps 5 and 6 of both cases, we see the $\pi/2$ fold of $D$ adds one unit of ribbonlength and two sticks. Thus the construction of the clasp region takes $4$ units of ribbonlength and $5$ sticks, regardless of the number of half-twists. 

Therefore in both cases, we can add the $n$ half-twists and clasp together to deduce we use $n+2+4=n+6$ units of folded ribbonlength and $n+3+5=n+8$ sticks. 
\end{proof}

As the figure-8 knot is a $T_2$ twist knot, we immediately get the following corollary. 
\begin{corollary} \label{cor:fig8}
There is a folded ribbon figure-8 knot $K_w$ with folded ribbonlength $\Rib(K_w)=8$.
\end{corollary} 

This result shows the folded ribbonlength has been drastically reduced from the previous best upper bound of $10$ from \cite{Den-FRF}, and gives evidence for the following conjecture. 

\begin{conjecture}\label{conj:fig8}
The infimal folded ribbonlength of a figure-8 knot $T_2$ is $\Rib([T_2])=8$.
\end{conjecture}

Just as we did with the $(2,q)$ torus knots, we can combine Theorem~\ref{thm:twist} with uniform folded ribbonlength upper bounds for twist knots that the second and fourth author proved in  \cite{Den-TP}.

\begin{corollary} \label{cor:twist}
For any twist knot $T_n$, the infimal ribbonlength 
$$\Rib([T_n])\leq \begin{cases} n+6  \quad & \text{ for $n\leq 9$ and $n=11$},
\\ 8\sqrt{3}+2\leq  15.86 & \text{ when $n\geq 10$ is even},
\\ 9\sqrt{3}+2\leq 17.59 & \text{ when $n\geq 13$ is odd}. 
 \end{cases}$$
\end{corollary}

Once again, the key observation is that the folding techniques which give uniform upper bounds for twist knots are not the same folding techniques for small crossing knots. As with Conjecture~\ref{conj:torus}, {\bf we conjecture the bounds given in Corollary~\ref{cor:twist} are the infimal folded ribbonlength bounds.} 

Finally, the crossing number of a $T_n$ twist knot is $\Cr(T_n)=n+2$. We again recover similar results for the folded ribbonlength crossing number conjecture as for the $(2,q)$-torus link case. For any twist knot $T_n$, the infimal ribbonlength $\Rib([T_n]) \leq \Cr(T_n) + 4$. This is again much lower than the universal upper bound for any knot type of $\Rib([K])\leq 2.5\Cr(K)+1$ given in Equation~\ref{eq:bound}.   For the lower bound, we again find that $\alpha=0$ in Equation~\ref{eq:crossing}.

\section{Ribbonlength of pretzel links}\label{sect:pretzel}

We are able to construct a folded ribbon 3-strand pretzel link $P(p,q,r)$ by using a modified version of the wrap method described in Construction~\ref{const:wrap}. Unlike the $(2,q)$-torus links and twist knots cases, we observe that $p$, $q$ and $r$ can be any integer.

\begin{const}[Modified wrap method for pretzels]\label{const:pretzel-wrap}
To construct $n$ half-twists, start with two pieces of ribbon labeled $AB$ and $CD$ and place ribbon $AB$ on top of ribbon $CD$.
There are several cases, depending on whether $n$ is even or odd, or positive or negative.
\\{\bf Case 1:} Let $n=0$ (no half-twists). Leave ribbon $AB$ on top of ribbon $CD$.
\\{\bf Case 2:} Let $n>0$ be an even integer.
\\ Step 1. Construct $n$ half-twists using the wrap method from Construction~\ref{const:wrap}. This is shown in the first image in Figure~\ref{fig:wrap-pretzel}. Recall that the half-twists are concentrated in a square region and end $B$ is pointing down at the front.
\\ Step 2. Fold end $B$ to the right with fold-angle $\nicefrac{\pi}{2}$  so that end $B$ lies over end $D$. This is shown in the second image in Figure~\ref{fig:wrap-pretzel}.
\begin{center}
    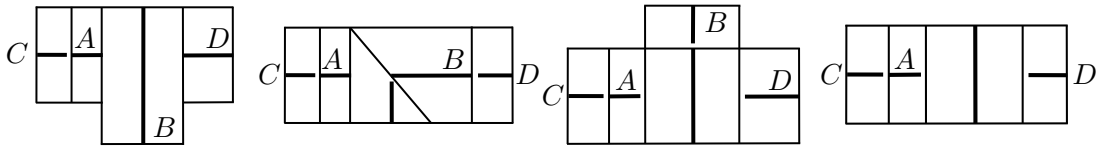
\begin{figure}[htpb]
    \begin{overpic}{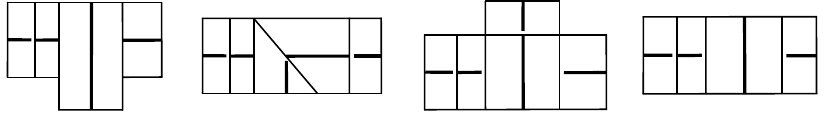}
\put(4.5, 9.5){$A$}
\put(12,1){$B$}
\put(-2,8){$C$}
\put(17,9.5){$D$}
\put(28, 7.5){$A$}
\put(39.5,7.5){$B$}
\put(22,6){$C$}
\put(46.5,6){$D$}
\put(56, 5.5){$A$}
\put(64.5,11){$B$}
\put(49,4){$C$}
\put(70.5,5.5){$D$}
\put(82.5, 7.5){$A$}
\put(75.5,6){$C$}
\put(99.5,6){$D$}
    \end{overpic}
    \caption{The modified wrap construction needed for pretzel links.}
    \label{fig:wrap-pretzel}
    \end{figure}
\end{center}
\noindent {\bf Case 3:} Let $n>0$ be an odd integer.
\\ Step 1. Construct $n$ half-twists using the wrap method from Construction~\ref{const:wrap}. This is shown in the third image in Figure~\ref{fig:wrap-pretzel}. We again note that the half-twists are concentrated in a square region and end $B$ is pointing up at the back.
\\ Step 2. Fold end $B$ at the back to the right with fold-angle $\nicefrac{\pi}{2}$  so that end $B$ lies under end $D$. This is shown in the rightmost image in Figure~\ref{fig:wrap-pretzel}. 

In both Case 2 and Case 3, we can construct negative half-twists using this method. Follow the same process but begin  Step 1 by folding end $B$ upward instead of downward. (See Construction~\ref{const:wrap}.)
We observe that at the end of this construction, the half-twists are concentrated in a square region. 
\qed
\end{const}
\begin{remark}\label{rmk:pretzel}
We follow the convention that the top (respectively bottom) ends of the $n$ half-twists constructed using the modified wrap method for pretzels correspond to the left (respectively right) ends of a strand of  $n$ half-twists in their standard knot diagrams.  We encourage the reader to make some examples of half-twists, then loosen the construction so the half-twists are separated. This convention will then become clear.
\end{remark}

\begin{center}
    \begin{figure}[htpb]
    \begin{overpic}{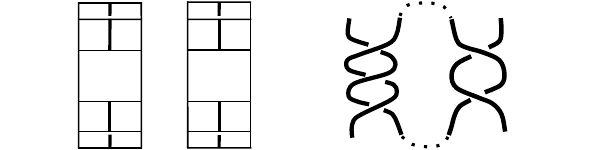}
    \put(16,11){$+3$}
    \put(34,11){$-2$}
    \put(13.5,18){$E$}
    \put(24,22){$F$}
    \put(13.5,4){$G$}
    \put(24,0){$H$}
    \put(32,18){\small{$M$}}
    \put(42.5,22){$N$}
    \put(32,4){$P$}
    \put(42.5,0){$Q$}
    \put(54,21){$E$}
    \put(63,22){$F$}
    \put(54.5,2){$G$}
    \put(63,0){$H$}
    \put(76,22){\small{$M$}}
    \put(84,22){$N$}
    \put(75,0){$P$}
    \put(86,2){$Q$}
        \end{overpic}
    \caption{The top ends of the modified wrap (on the left) correspond to the left ends of a standard strand of half-twists (on the right). Similarly the bottom ends of the modified wrap correspond to the right ends of the half-twists.}
    \label{fig:pretzel-twists}
    \end{figure}
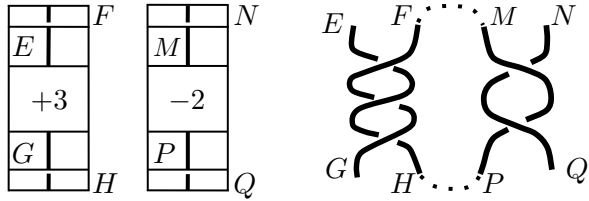
\end{center}

An example of the convention in Remark~\ref{rmk:pretzel} can be seen in Figure~\ref{fig:pretzel-twists}, where we see two modified wraps and their corresponding knot diagrams. Observe that the corresponding ends for each of these constructions have been labeled accordingly. Ends $E$, $G$, $M$, and $P$ are the top and left ends of the modified wraps and the standard knot diagrams, respectively, while ends $F$, $H$, $N$, and $Q$ are the bottom and right ends, respectively. 

When constructing a $P(p,q,r)$ pretzel link we join the three strands of $p$, $q$ and $r$ half-twists in a particular way (see Section~\ref{sect:knots}). Namely, the right ends of the $p$ half-twists join the left ends of the $q$ half-twists, and so on. The dotted lines in the right image in Figure~\ref{fig:pretzel-twists} show an example of this. What about the corresponding modified wraps? The convention in Remark~\ref{rmk:pretzel} means that we take take the two pieces of ribbon in the left image in Figure~\ref{fig:pretzel-twists}, and place the ribbon with $+3$ half-twists over the ribbon with $-2$ half-twists. We then join ends $F$ and $M$ with a fold line and join ends $H$ and $P$ with a fold line. The ribbon in each of these joins can be shortened so the fold lines lie against the square region created by the half-twists. 

The final remark to make is that when constructing a $P(p,q,r)$ pretzel link, it might be the case that one of $p$, $q$ or $r$ is $0$. In this case there are no half-twists and the corresponding knot diagram has two parallel lines. As discussed in Construction~\ref{const:pretzel-wrap}, we take two pieces of ribbon and simply place one piece of ribbon over the other.  We now put all of these observations together in the following construction.

\begin{const}[Folded ribbon $P(p,q,r)$ pretzel link]\label{const:pretzel}
To construct a folded ribbon $P(p,q,r)$ pretzel link, begin by using Construction~\ref{const:pretzel-wrap} to fold three modified wraps with $p$, $q$, and $r$ half-twists.  We let {\em wrap $p$} denote the ribbons with $p$ half-twists. 
\\Step 1: Arrange these three wraps so that wrap $p$ lies over wrap $q$, which in turn lies over wrap $r$. Recall that in each wrap, the half-twists occur in a square region.

\begin{center}
    \begin{figure}[htpb]
    \begin{overpic}{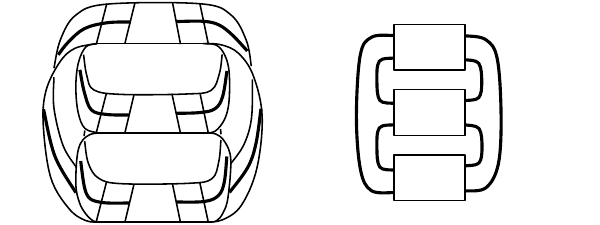}
        \put(24,34.5){\rotatebox{270}{$p$}}
        \put(24,19.5){\rotatebox{270}{$q$}}
        \put(24,4.5){\rotatebox{270}{$r$}}
        \put(70.5,30.5){\rotatebox{270}{$p$}}
        \put(70.5,19.5){\rotatebox{270}{$q$}}
        \put(70.5,9){\rotatebox{270}{$r$}}
    \end{overpic}
    \caption{The left image shows Step 2 of the construction of a $P(p,q,r)$ pretzel link. The right image shows the corresponding link diagram.}
    \label{fig:pretzel-const}
    \end{figure}
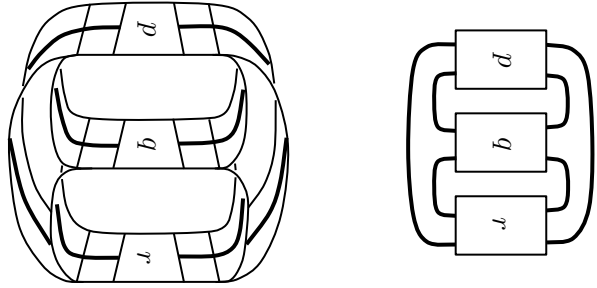
\end{center}
\noindent Step 2: Following Remark~\ref{rmk:pretzel}, connect the top ends of wrap $r$ to their corresponding bottom ends of wrap $q$ with a join~line. In the same fashion, connect the top ends of wrap $q$ to the bottom ends of wrap $p$ with a join~line. Then connect the bottom ends of wrap $r$ to their corresponding top ends of wrap $p$ with a join~ line.  This is illustrated on the left in Figure~\ref{fig:pretzel-const}. The right image shows the relationship between the ribbon construction and the link diagram for the $P(p,q,r)$ pretzel link.
\\ Step 3: Minimize the length of these connections until all of the join lines lie against the square regions created by the half-twists. The folded ribbon pretzel link is now in the shape of a square.  
\qed
\end{const}

\begin{theorem}\label{thm:pretzel}
Any $(p,q,r)$-pretzel link type $L$ contains a folded ribbon pretzel link $L_w$ with folded ribbonlength $\Rib(L_w) = |p|+|q|+|r|+6.$ The corresponding link diagram has $|p|+|q|+|r|+12$ sticks.
\end{theorem}
\begin{proof} We construct a folded ribbon $(p,q,r)$-pretzel link using Construction~\ref{const:pretzel} and find its folded ribbonlength. We began Construction~\ref{const:pretzel} by constructing three modified wraps with $p$, $q$, and $r$ half-twists. We know from Lemma~\ref{lem:twist} that the folded ribbonlength of the unmodified wrap method for $n$ half-twists needs at least $|n|+2$ units of ribbonlength, and the corresponding knot diagram has $|n|+3$ sticks. The only difference between the two wrap constructions is that the final square of ribbon is folded with fold angle $\pi/2$ to give two overlapping isosceles right triangles.  Recall from Remark~\ref{rmk:length}, that both a square and a pair of isosceles right triangles have ribbonlength 1.  Hence the ribbonlength required to construct one modified wrap of $n$ half-twists is $|n|+2$. The corresponding knot diagram needs $|n|+4$ sticks, as one extra stick was constructed when the final square of ribbon was folded.

Steps 2 and 3 of Construction~\ref{const:pretzel} describe how these three modified wraps are joined with no additional ribbonlength to complete the pretzel link. Therefore the folded ribbonlength for this geometric realization is $\Rib(L_w) =(|p|+2)+(|q|+2)+(|r|+2)=|p|+|q|+|r|+6$. The corresponding link diagram uses $|p|+|q|+|r|+12$ sticks.
\end{proof}

We compare Theorem~\ref{thm:pretzel} with the previous work from \cite{Den-FRF}. There, we found the infimal folded ribbonlength of any $P(p,q,r)$ pretzel link is $\Rib([P(p,q,r)]) \leq 2(|p|+|q|+|r|)+2$. At first glance, our result of $\Rib([P(p,q,r)]) \leq |p|+|q|+|r|+6$ improves on this for values of $p$, $q$, and $r$ where $|p|+|q|+|r|>3$.  Let us consider the knot and link types when $|p|+|q|+|r| = 3$. The case  $P(\pm 1, \pm 1, \pm1)$ yields either a trefoil knot or an unknot. The case $P(\pm2,\pm1\pm 0)$ yields a Hopf link, while the case $P(\pm 3, 0, 0)$ yields a split link made up of a trefoil knot and an unknot. If we consider the cases when $|p|+|q|+|r| < 3$ we find unknots and unlinks. We know the infimal ribbonlength of an unknot is 0, while the infimal ribbonlength of a trefoil knot is $\leq 6$ and the infimal ribbonlength of a Hopf link\footnote{Draw a Hopf link diagram so that each component is made from two sticks. Thus each component  of a folded ribbon Hopf link is constructed from two square pieces of ribbon and has folded ribbonlength 2.} is $\leq 4$. We thus deduce that in each of these cases the infimal ribbonlength is $< |p|+|q|+|r|+6$. In summary, we find the following.

\begin{corollary}\label{cor:pretzel}
For any pretzel link $P(p,q,r)$, the infimal folded ribbonlength $$\Rib([P(p,q,r)]) \leq |p|+|q|+|r|+6.$$
\end{corollary}

In Section~\ref{sect:knots}, we noted that 3-strand pretzel links can be generalized to $n$-strand pretzel links made with $(p_1,p_2,\dots, p_n)$ half-twists.  Construction~\ref{const:pretzel} and Theorem~\ref{thm:pretzel} can immediately be generalized to this case. 

\begin{corollary}\label{n-strand}
For any $n$-strand pretzel link $L$ constructed $(p_1,p_2,\dots,p_n)$ half-twists, the infimal folded ribbonlength 
$\Rib([L])\leq(\sum_{i=1}^n|p_i|)+2n$.
\end{corollary}

Recall from Section~\ref{sect:knots}, that any twist knot with $n$ half-twists can be constructed as a $P(n,1,1)$ pretzel link. Thus, a twist knot $K$ made as a pretzel knot (from Construction~\ref{const:pretzel}) has folded ribbonlength $\Rib(K_w) = |n|+|1|+|1|+6=n+8$. This does not improve on the ribbonlength of $\Rib(K_w)=n+6$ from Theorem~\ref{thm:twist}.

What about crossing number? It is hard to compute the crossing number of a $P(p,q,r)$ pretzel link. In general, we only know that $\Cr(P(p,q,r))\leq |p| + |q| +|r|$. We do know that when $p$, $q$, and $r$ are all odd and have the same sign that $\Cr(P(p,q,r)) = |p| + |q| +|r|$. In this special case, Theorem~\ref{thm:pretzel} implies that $\Rib([P(p,q,r)])\leq \Cr(P(p,q,r))+6$.  Once again, this is less than the universal upper bound for any knot type of $\Rib([K])\leq 2.5\Cr(K)+1$ given in Equation~\ref{eq:bound}.   What about the lower bound of the ribbonlength crossing number conjecture in Equation~\ref{eq:crossing}?  In~\cite{Den-TP}, the second and fourth authors show that any number of half-twists can be constructed with a finite amount of folded ribbonlength. There is also a finite amount of ribbonlength needed to connect the ends of the three strands of half-twists in order to construct a pretzel link. Thus in \cite{Den-TP} we conjectured the following.

\begin{conjecture}\label{conj:pretzel} For all $(p,q,r)$ pretzel links $L$, there is a constant $C>0$, such that the folded ribbonlength
$\Rib(L_w)\leq C$.
\end{conjecture}

As with the previous sections, we do not expect the constant $C$ to be the best possible upper bound for pretzel links with small crossing number. Instead, Corollary~\ref{cor:pretzel} will apply. Just as Corollary~\ref{n-strand} generalizes Corollary~\ref{cor:pretzel}, we expect a version of Conjecture~\ref{conj:pretzel} to hold for $n$-strand pretzel links for each $n$. 

We are now in the position of giving good upper bounds on the infimal folded ribbonlength for many small crossing knots and links. Our results are summarized in the following table, where we have used the knot and link classification found in KnotInfo \cite{knotinfo} and LinkInfo \cite{linkinfo}. 
\begin{center}
\begin{tabular}{ | c | c | c | c |}
\hline
Knot/link table& Name & $\Rib([K])\leq$ ??? & Notes 
\\ \hline
$0_1$ & unknot & 0 & 
\\ \hline
$3_1$ & trefoil & 6 &  $\Rib[T(2,3)]\leq 3+3= 6$
\\ \hline
$4_1$ & figure-8 & 8 & $\Rib[(T_2)]\leq 2+6 = 8$
\\ \hline
$5_1$ & cinquefoil & 8 & $\Rib[T(2,5)]\leq 5+3= 8$
\\ \hline
$5_2$&  3-twist & 9 & $\Rib[(T_3)]\leq 3+6 = 9$
\\ \hline
$6_1$&  4-twist & 10 & $\Rib[(T_4)]\leq 4+6 = 10$
\\ \hline
$6_2$& $P(1,2,3)$ & 12 & $\Rib[P(1,2,3)]\leq 1+2+3+6 = 12$
\\ \hline
$6_3$&  $P(2,1,-3,1)$ & 15 & $\Rib[P(2,1,-3,1)]\leq 2+1+3+1+8 = 15$
\\ \hline
L2a1 & Hopf link & 4 & $\Rib[\text{Hopf}]\leq 4$
\\ \hline
L4a1 & $(2,4)$-torus & 7 & $\Rib([T(2,4)])\leq 4+3=7$
\\ \hline
L6a3 & $(2,6)$-torus & 9 & $\Rib([T(2,6)])\leq 6+3=9$
\\ \hline
\end{tabular}
\end{center}


\section{Acknowledgments}

We would like to thank the generous support of Washington \& Lee University (W\&L). The first, third, and fourth authors' research was funded by W\&L's 2024 Summer Research Scholars Program. The second author's research was funded by a 2024 Lenfest grant from W\&L.  The folded ribbon knot project has developed over multiple years. The work we have done here builds on the work done by the second author's prior collaborations, especially the work done with John Carr Haden and Troy Larsen. Thank you.

All the figures in this paper were made using Google Draw. 

\bibliography{folded-ribbons}{}

\begin{thebibliography}{10}

\bibitem{Adams}
Colin~C. Adams.
\newblock {\em The knot book}.
\newblock W. H. Freeman and Company, New York, 1994.
\newblock An elementary introduction to the mathematical theory of knots.

\bibitem{RES-Brown}
Brienne~Elisabeth Brown and Richard~Evan Schwartz.
\newblock The crisscross and the cup: Two short 3-twist paper moebius bands,
  2023.
\newblock \url{https://arxiv.org/abs/2310.10000}.

\bibitem{BS99}
Gregory Buck and Jonathan Simon.
\newblock Thickness and crossing number of knots.
\newblock {\em Topology Appl.}, 91(3):245--257, 1999.

\bibitem{CKS}
Jason Cantarella, Robert~B. Kusner, and John~M. Sullivan.
\newblock On the minimum ropelength of knots and links.
\newblock {\em Invent. Math.}, 150(2):257--286, 2002.

\bibitem{Crom}
Peter~R. Cromwell.
\newblock {\em Knots and links}.
\newblock Cambridge University Press, Cambridge, 2004.

\bibitem{CR}
H.~Martyn Cundy and A.~P. Rollett.
\newblock {\em Mathematical models}.
\newblock Oxford, at the Clarendon Press, 1952.

\bibitem{Den-FRS}
Elizabeth Denne.
\newblock Folded ribbon knots in the plane.
\newblock In Colin Adams, Erica Flapan, Allison Henrich, Louis Kauffman, Lew
  Ludwig, and Sam Nelson, editors, {\em A Concise Encyclopedia of Knot Theory},
  chapter~88, pages 877--897. Chapman and Hall/CRC, Boca Raton, FL, 2021.

\bibitem{Den-FRC}
Elizabeth Denne.
\newblock Ribbonlength and crossing number for folded ribbon knots.
\newblock {\em J. Knot Theory Ramifications}, 30(4):2150028, 21, 2021.

\bibitem{DDS}
Elizabeth Denne, Yuanan Diao, and John Sullivan.
\newblock Quadrisecants give new lower bounds for the ropelength of a knot.
\newblock {\em §Geom. Topol.}, 10:1--26, 2006.

\bibitem{Den-FRF}
Elizabeth Denne, John~Carr Haden, Troy Larsen, and Emily Meehan.
\newblock Ribbonlength of families of folded ribbon knots.
\newblock {\em Involve}, 15(4):591--628, 2022.

\bibitem{DKTZ}
Elizabeth Denne, Mary Kamp, Rebecca Terry, and Xichen Zhu.
\newblock Ribbonlength of folded ribbon unknots in the plane.
\newblock In {\em Knots, links, spatial graphs, and algebraic invariants},
  volume 689 of {\em Contemp. Math.}, pages 37--51. Amer. Math. Soc.,
  Providence, RI, 2017.

\bibitem{Den-FRLU}
Elizabeth Denne and Troy Larsen.
\newblock Linking number and folded ribbon unknots.
\newblock {\em J. Knot Theory Ramifications}, 32(1):Paper No. 2350003, 42,
  2023.

\bibitem{Den-TP}
Elizabeth Denne and Timi Patterson.
\newblock Bounded ribbonlength for knot families and multi-twist m\"obius
  bands, 2025.
\newblock \url{https://arxiv.org/abs/2509.18370}.

\bibitem{DE}
Yuanan Diao and Claus Ernst.
\newblock The ropelengths of alternating montesinos links are proportional to
  their crossing numbers.
\newblock {\em Journal of Physics A: Mathematical and Theoretical},
  58(33):335202, aug 2025.

\bibitem{Flap}
Erica Flapan.
\newblock {\em When Topology Meets Chemistry: A Topological Look at Molecular
  Chirality}.
\newblock Outlooks. Cambridge University Press, 2000.

\bibitem{gm}
Oscar Gonzalez and John~H. Maddocks.
\newblock Global curvature, thickness, and the ideal shapes of knots.
\newblock {\em Proc. Natl. Acad. Sci. USA}, 96(9):4769--4773, 1999.

\bibitem{HW}
B.~Halpern and C.~Weaver.
\newblock Inverting a cylinder through isometric immersions and isometric
  embeddings.
\newblock {\em Trans. Amer. Math. Soc.}, 230:41--70, 1977.

\bibitem{Hen}
Aidan Hennessey.
\newblock Constructing many-twist {M\"obius} bands with small aspect ratios.
\newblock {\em Comptes Rendus. Math\'ematique}, 362:1837--1845, 2024.

\bibitem{John}
D.A. Johnson.
\newblock {\em Paper {F}olding for the {M}athematics {C}lass}.
\newblock Washington D.C. National Council of Teachers of Mathematics, 1957.

\bibitem{JohnHen}
Inga Johnson and Allison Henrich.
\newblock {\em An {I}nteractive {I}ntroduction to {K}not {T}heory}.
\newblock Dover Publications, 2017.

\bibitem{Kauf05}
Louis~H. Kauffman.
\newblock Minimal flat knotted ribbons.
\newblock In {\em Physical and numerical models in knot theory}, volume~36 of
  {\em Ser. Knots Everything}, pages 495--506. World Sci. Publ., Singapore,
  2005.

\bibitem{KMRT}
Brooke Kennedy, Thomas~W. Mattman, Roberto Raya, and Dan Tating.
\newblock Ribbonlength of torus knots.
\newblock {\em J. Knot Theory Ramifications}, 17(1):13--23, 2008.

\bibitem{KNY-TwTorus}
Hyoungjun Kim, Sungjong No, and Hyungkee Yoo.
\newblock Ribbonlength of twisted torus knots.
\newblock {\em J. Knot Theory Ramifications}, 31(12):Paper No. 2250092, 10,
  2022.

\bibitem{KNY-2Bridge}
Hyoungjun Kim, Sungjong No, and Hyungkee Yoo.
\newblock Folded ribbonlength of 2-bridge knots.
\newblock {\em J. Knot Theory Ramifications}, 32(4):Paper No. 2350030, 11,
  2023.

\bibitem{KNY-Lin}
Hyoungjun Kim, Sungjong No, and Hyungkee Yoo.
\newblock Linear upper bounds on the ribbonlength of knots and links, 2024.
\newblock \url{https://arxiv.org/abs/2409.13572}.

\bibitem{DNA1}
Nicole C~H Lim and Sophie~E Jackson.
\newblock Molecular knots in biology and chemistry.
\newblock {\em Journal of Physics: Condensed Matter}, 27(35):354101, aug 2015.

\bibitem{lsdr}
R.~A. Litherland, J.~Simon, O.~Durumeric, and E.~Rawdon.
\newblock Thickness of knots.
\newblock {\em Topology Appl.}, 91(3):233--244, 1999.

\bibitem{Liv}
Charles Livingston.
\newblock {\em Knot theory}, volume~24 of {\em Carus Mathematical Monographs}.
\newblock Mathematical Association of America, Washington, DC, 1993.

\bibitem{knotinfo}
Charles Livingston and Allison~H. Moore.
\newblock Knotinfo: Table of knot invariants.
\newblock URL: \url{knotinfo.org}, October 16 2025.

\bibitem{linkinfo}
Charles Livingston and Allison~H. Moore.
\newblock Linkinfo: Table of link invariants.
\newblock URL: \url{knotinfo.org}, October 16 2025.

\bibitem{RES-Mont}
Noah Montgomery and Richard~Evan Schwartz.
\newblock The optimal twisted paper cylinder, 2025.
\newblock \url{https://arxiv.org/abs/2309.14033}.

\bibitem{RiboRobot}
Sebastian Risi, Daniel Cellucci, and Hod Lipson.
\newblock Ribosomal robots: evolved designs inspired by protein folding.
\newblock In {\em Proceedings of the 15th Annual Conference on Genetic and
  Evolutionary Computation}, GECCO '13, page 263–270, New York, NY, USA,
  2013. Association for Computing Machinery.

\bibitem{RES-Mob2}
Richard~Evan Schwartz.
\newblock On nearly optimal paper {M}oebius bands.
\newblock {\em Adv. Geom.}, 25(2):207--214, 2025.

\bibitem{RES-Mob}
Richard~Evan Schwartz.
\newblock The optimal paper {M}oebius band.
\newblock {\em Ann. of Math. (2)}, 201(1):291--305, 2025.

\bibitem{Tian-A}
Grace Tian.
\newblock Linear upper bound on the ribbonlength of torus knots and twist
  knots, 2018.
\newblock \url{https://arxiv.org/abs/1809.02095}.

\bibitem{RobFold}
Liyu Wang, Mark~M. Plecnik, and Ronald~S. Fearing.
\newblock Robotic folding of 2d and 3d structures from a ribbon.
\newblock In {\em 2016 IEEE International Conference on Robotics and Automation
  (ICRA)}, pages 3655--3660, 2016.

\bibitem{Wel}
David Wells.
\newblock {\em The {P}enguin dictionary of curious and interesting geometry}.
\newblock Penguin Books, New York, 1991.

\end{thebibliography}
\bibliographystyle{plain}


\end{document}